\documentclass[11pt]{amsart}
\usepackage{amsmath, amssymb}
\usepackage{amsfonts}
\usepackage{mathrsfs}
\usepackage{adjustbox}
\usepackage[arrow,matrix,curve,cmtip,ps]{xy}

\usepackage{amsthm}
\usepackage{color}

\usepackage{verbatim}
 
\newtheorem{theorem}{Theorem}[section]
\newtheorem{lemma}[theorem]{Lemma}
\newtheorem{proposition}[theorem]{Proposition}
\newtheorem{corollary}[theorem]{Corollary}

\newtheorem*{theorem*}{Theorem}
\theoremstyle{remark}

\newtheorem{definition}[theorem]{Definition}
\newtheorem{example}[theorem]{Example}

\newcommand\bgb[1]{{#1}}

\numberwithin{equation}{section}
 
\newcommand{\Z}{\mathbb{Z}}
\newcommand{\N}{\mathbb{N}}
\newcommand{\R}{\mathbb{R}}

\newcommand{\F}{\mathcal{F}}
\newcommand{\G}{\mathcal{G}}

\newcommand{\gae}{\lower 2pt \hbox{$\, \buildrel {\scriptstyle >}\over {\scriptstyle
\sim}\,$}}
 
\newcommand{\lae}{\lower 2pt \hbox{$\, \buildrel {\scriptstyle <}\over {\scriptstyle
\sim}\,$}}
 
\newcommand{\MU}[1]{
\setbox0\hbox{$#1$}
\setbox1\hbox{$W$}
\ifdim\wd0>\wd1 #1^{\sim} \else \widetilde{#1} \fi
}

\begin{document}
\title[What is odd about binary Parseval frames?]{What is odd about binary Parseval frames?}
 
 \author[Z. J. Baker]{Zachery J. Baker}
\author[B. G. Bodmann]{Bernhard G. Bodmann}
\author[M. G. Bullock]{Micah G. Bullock}
\author[S. N. Branum]{Samantha N. Branum}
\author[J. McLaney]{Jacob E. McLaney}
 
\address{Department of Mathematics \\ University of Houston \\ Houston, TX 77204-3008 \\USA}
\email{zacherybaker96@gmail.com}
\email{bg{}b{}@{}math{}.uh.{}edu}
\email{micahbullock@outlook.com}
\email{sambranum@gmail.com}
\email{mclaneyjacob@gmail.com}

\thanks{This research was supported by NSF grant DMS-1109545 and DMS-1412524}
 
\date{\today}
 
\subjclass[2010]{42C15}
 
\keywords{frames, Parseval frames, finite-dimensional vector spaces, binary numbers, orthogonal extension principle,
switching equivalence, Gram matrices}
 
 
\begin{abstract}
This paper examines the construction and properties of binary Parseval frames.
We address two questions: When does a binary Parseval frame have
a complementary Parseval frame? Which binary symmetric idempotent matrices
are Gram matrices of binary Parseval frames? In contrast to the case of real or complex 
Parseval frames, the answer to these questions is not always affirmative. The key
to our understanding comes from an algorithm that constructs
binary orthonormal sequences that span a given subspace, whenever possible.
Special regard is given to binary frames whose Gram matrices are circulants.
\end{abstract}
 
\maketitle

\section{Introduction}

Much of the literature on frames, from its beginnings in non-harmonic Fourier analysis \cite{DS} to
comprehensive overviews of theory and applications \cite{Chr03,CK07a,CK07b} assumes an underlying structure of a real or complex Hilbert space
to study approximate expansions of vectors. Indeed, the correspondence between
vectors in Hilbert spaces and linear functionals given by the Riesz representation theorem provides a convenient
way to characterize Parseval frames, sequences of vectors that behave in a way that is similar
to orthonormal bases without requiring the vectors to be linearly independent~\cite{Chr03}.
Incorporating linear dependence relations is useful to permit more flexibility for expansions
and to suppress errors that may model faulty signal transmissions in applications \cite{Marshall84,Marshall89,RG03,RG04,HP04,PK05,BP05}.

The concept of frames has also been established even in vector spaces without
(definite) inner product \cite{B09,HKLW}.  In fact, the well known theory of binary codes could be seen as
a form of frame theory, in which linear dependence relations among binary vectors 
are examined \cite{MWS77,HPvR99,BBFKKW}. Here, binary vector spaces are defined over the finite field
with two elements; a frame for a finite dimensional binary vector space is simply a spanning sequence \cite{B09}. In a preceding paper \cite{B14}, the study of binary codes from a frame theoretic perspective has
lead to additional combinatorial insights in the design of error-correcting codes.

The present paper is concerned with binary Parseval frames. These binary frames provide explicit
expansions of binary vectors using a bilinear form that resembles the
dot product in Euclidean spaces. In contrast to the inner product on real or complex
Hilbert spaces, there are many nonzero vectors whose dot product with themselves vanishes. 
Such vectors have special significance in our results. Counting the number of non-vanishing entries  
motivates calling them \emph{even} vectors, whereas the others are called \emph{odd}.
As a consequence of the degeneracy of the bilinear form, there are some striking differences 
with frame theory over real or complex Hilbert spaces. In this paper, we explore
the construction and properties of binary Parseval frames, and 
compare them with real and complex ones. Our main results are as follows:

In the real or complex case, it is known that each Parseval frame has a Naimark complement \cite{Chr03}.
The complementarity is most easily formulated by stating that the Gram matrices of two complementary
Parseval frames sum to the identity. We show that in the binary case, not every Parseval frame
has a Naimark complement. In addition, we show that a necessary and sufficient condition
for its existence is that the Parseval frame contains at least one even vector.

Moreover, we study the structure of Gram matrices. 
The Gram matrices of real or complex Parseval frames are characterized as symmetric or hermitian
idempotent matrices. The binary case requires the additional condition that at least one column vector
of the matrix is odd.

The general results we obtain are illustrated with examples. Special regard is given to 
cyclic binary Parseval frames, whose Gram matrices are circulants.
 
\section{Preliminaries}\label{sec:prelim}

\bgb{We define the notions of a binary frame and a binary Parseval frame as in a previous paper \cite{B09}. The vector space that these sequences of vectors span is the direct sum $\Z_2^n = \Z_2 \oplus \dotsb \oplus \Z_2$ of $n$ copies of $\Z_2$ for some $n \in \N$. Here, $\Z_2$ is the field of binary numbers 
with the two elements $0$ and $1$, the neutral element with respect to addition and the multiplicative identity.}

\begin{definition}
A \emph{binary frame} is a sequence $\mathcal F = \{f_1 , \ldots, f_k\}$ in a binary vector space $\mathbb Z_2^n$ such that
$
\mathrm{span}\,{\F} = \mathbb Z_2^n.
$
\end{definition}

\bgb{A simple example of a frame is the canonical basis $\{e_1, e_2, \dots, e_n\}$ for $\Z_2^n$. The $i$th vector
has components  $(e_i)_j = \delta_{i,j}$, thus $(e_i)_i=1$ is the only non-vanishing entry for $e_i$. Consequently, a vector $x=(x_i)_{i=1}^n$ is expanded in terms of the canonical basis
as $x = \sum_{i=1}^n x_i e_i$. }

Frames provide similar expansions of vectors in linear combinations of the frame vectors. Parseval frames are especially convenient for this purpose,
because the linear combination can be determined with little effort. In the real or complex case, this only requires computing values 
of inner products between the vector to be expanded and the frame vectors. Although we cannot introduce a
non-degenerate inner product in the binary case, we define Parseval frames using a bilinear form that resembles the dot product on $\R^n$. Other choices of bilinear forms and a more general theory of binary frames have been investigated elsewhere, see \cite{HLS}.

\begin{definition}
\bgb{The \emph{dot product} on $\mathbb Z_2^n$ is the bilinear map $( \cdot, \cdot ) : \Z_2^n \times \Z_2^n \to \Z_2$ given by}
\[
\left( \left( \begin{smallmatrix} x_1 \\ \vdots \\ x_n \end{smallmatrix} \right),  \left( \begin{smallmatrix} y_1 \\ \vdots \\ y_n \end{smallmatrix} \right) \right) := \sum_{i=1}^n x_i y_i.
\]
\end{definition}

With the help of this dot product, we define a Parseval frame for $\mathbb{Z}_2^n$.

\begin{definition} \label{def:pframe}
A \emph{binary Parseval frame} is a sequence of vectors $\F = \{ f_1, \ldots$, $f_k \}$ in $\Z_2^n$ such that 
for all $x \in \Z_2^n$, the sequence satisfies the reconstruction identity
\begin{equation} \label{Reconstruction-Formula-Z2}
x = \sum_{j=1}^k ( x, f_j )f_j  \, .
\end{equation}
\bgb{To keep track of the specifics of such a Parseval frame, we then also say that $\F$ is a \emph{binary $(k,n)$-frame}.}
\end{definition}

In the following, we use matrix algebra whenever it is convenient for establishing properties of frames. We write $A \in M_{m,n} (\Z_2)$ when $A$ an $m \times n$ matrix with entries in $\Z_2$ and identify $A$ with \bgb{the} linear map from $\Z_2^n$ to $\Z_2^m$ \bgb{induced  by left multiplication of any (column) vector $x \in \Z_2^n$ with $A$.}
We let $A^*$ denote the adjoint of $A\in M_{m,n}(\Z_2)$; that is, $(Ax, y) =(x,A^* y)$ for all $x\in \Z_2^n, y \in \Z_2^m$
and consequently, $A^*$ is the transpose of $A$.   

\begin{definition}
Each frame $\F = \{ f_1, \ldots, f_k \}$ is associated with its \emph{analysis matrix} $\Theta_\F$,
whose $i$th row is given by the $i$th frame vector for $i \in \{1, 2, \dots, k\}$.
Its transpose $\Theta_\F^*$ is called the \emph{synthesis matrix}.
\end{definition}



With the help of matrix multiplication, the reconstruction formula (\ref{Reconstruction-Formula-Z2})
of a binary $(k,n)$-frame $\F$  with analysis matrix $\Theta_\F$ is simply expressed as
\begin{equation} \label{eq:recidmat}
   \Theta^*_\F \Theta_\F = I_n \, ,
\end{equation}
where $I_n$ is the $n\times n$ identity matrix. We also note that for any $x, y \in \Z_2^n$, then
$$
  \langle \Theta_\F x, \Theta_\F y \rangle = \langle x, y \rangle
$$
which motivates speaking of $\Theta_\F$ as an \emph{isometry}, as in the case of real or complex inner product spaces.

Another way to interpret identity~(\ref{eq:recidmat}) is in terms of the \emph{column} vectors of $\Theta_\F$.
Again borrowing a concept from Euclidean spaces, we introduce orthonormality.

\begin{definition}
We say that a sequence  of vectors $\{v_1, v_2, \dots, v_r\}$  in $\mathbb Z_2^n$ 
is \emph{orthonormal} if $(v_i, v_j) = \delta_{i,j}$ for $i, j \in \{1,2, \dots, r\}$, that  is, the 
dot product of the pair $v_i$ and $v_j$ vanishes unless $i=j$, in which case it is equal to one. 
\end{definition}

Inspecting the matrix identity (\ref{eq:recidmat}),
we see that
a binary $k \times n$ matrix $\Theta$ is the analysis matrix of a binary Parseval frame if
and only if
the columns of $\Theta$ form an orthonormal sequence in $\mathbb Z_2^k$.

The orthogonality relations between the frame vectors are recorded in the Gram matrix,
whose entries consists of the dot products of all pairs of vectors.

\begin{definition}
The \emph{Gram matrix} of a binary frame $\F=\{f_1, f_2, \dots, f_k\}$ for $\mathbb Z_2^n$ is the 
$k\times k$ matrix $G$ with entries
$ G_{i,j} = (f_j , f_i)$. 
\end{definition}

It is straightforward to verify that the Gram matrix of $\F$ is expressed as the composition
of the analysis and synthesis matrices,
$$
  G = \Theta_\F \Theta^*_\F \, .
$$
The identity (\ref{eq:recidmat}) implies that the Gram matrix of a Parseval frame satisfies the equations
$$
  G = G^* = G^2 \, .
$$ 
For frames over the real or complex numbers, these equations characterize the set of all Gram matrices of Parseval frames
as orthogonal projection matrices. 
However, in the binary case, this is only a necessary condition, as shown in the following proposition and the subsequent example.

\begin{proposition}\label{prop:MnotG}
If $M$ is binary matrix that satisfies $M=M^2=M^*$ and it has only even column vectors,
then $M$ is not the Gram matrix of a binary Parseval frame.
\end{proposition}
\begin{proof}
If $G$ is the Gram matrix of a Parseval frame with analysis
operator $\Theta$, then $G \Theta = \Theta \Theta^* \Theta = \Theta$,
and thus for each column $\omega$ of $\Theta$, we obtain the
eigenvector equation $G\omega=\omega$. By the orthonormality of the columns of $\Theta$, 
each $\omega$ is odd. 

On the other hand, if $M$ has only even columns, then any eigenvector
corresponding to eigenvalue one is even, because it is a linear combination of the column vectors of $M$. 
This means $M$ cannot be the Gram matrix of a binary Parseval frame.
\end{proof}

The following example shows that idempotent symmetric matrices that are not Gram matrices of binary Parseval frames exist for any odd dimension $k \ge 3$.

\begin{example}\label{ex:MnotG}
Let $k\ge 3$ be odd and let $M$ be the $k \times k$ matrix whose entries are all equal to one
except for vanishing entries on the diagonal,
$ M_{i,j} = 1 -\delta_{i,j} \, , i, j \in \{1, 2, \dots, k\}.
$ 
This matrix
satisfies $M=M^2 = M^*$, but only has even columns and 
by the preceding proposition, it is not the Gram matrix of a binary Parseval frame.
\end{example}

As shown in Section~\ref{sec:binaryGram}, having only even column vectors is the only way a binary symmetric idempotent
matrix can fail to be the Gram matrix of a Parseval frame. 
The construction of Example~\ref{ex:MnotG} is intriguing, because the alternative choice where $k$ is odd and all entries of $M$ are equal to one \emph{is} the Gram matrix of a binary Parseval frame. The relation between these two alternatives can be 
interpreted as complementarity, which will be explored in more detail in the next section.


\section{Complementarity for binary Parseval frames}

Over the real or complex numbers, each Parseval frame has a so-called Naimark complement \cite{Chr03}; if $G$ is the Gram matrix
of a real or complex Parseval frame, then it is an orthogonal projection matrix, and so is $I-G$, which makes it
the Gram matrix of a complementary Parseval frame.

We adopt the same definition for the binary case.

\begin{definition}
Two binary Parseval frames $\F$ and $\G$ having analysis operators $\Theta_\F \in M_{k,n}(\Z_2)$ and $\Theta_\G \in M_{k,k-n}(\Z_2)$ are \emph{complementary}
if 
$$   \Theta_\F \Theta_\F^* + \Theta_\G \Theta_\G^* = I_k \, .$$
We also say that $\F$ and $\G$ are \emph{Naimark complements} of each other.
\end{definition}

There is an equivalent statement of complementarity in terms of the block matrix $ U = (\Theta_\F \, \Theta_\G)$
formed by adjoining $\Theta_\F$ and $\Theta_\G$ being \emph{orthogonal}, meaning $U U^* = U^* U = I$,
just as 
in the real case (or as $U$ being unitary in the complex case).

\begin{proposition}\label{prop:compequiv}
Two binary Parseval frames $\F$ and $\G$ having analysis operators $\Theta_\F \in M_{k,n}(\Z_2)$ and $\Theta_\G \in M_{k,k-n}(\Z_2)$ are complementary
if and only if the block matrix  $ (\Theta_\F \, \Theta_\G)$ is an orthogonal $k \times k$ matrix. \end{proposition}
\begin{proof}
In terms of the block matrix $(\Theta_\F \, \Theta_\G)$, the complementarity is expressed as 
$$  (\Theta_\F \, \Theta_\G) (\Theta_\F \, \Theta_\G)^*  = I_k \, .$$
Since $U = (\Theta_\F \, \Theta_\G) $ is a square matrix, $U U^*=I$  is equivalent to $U^*$ also being 
a left inverse of $U$, meaning $U U^* = U^* U = I_k$, $U$ is orthogonal. 
\end{proof}

In the binary case, not every Parseval frame has a Naimark complement. For example, if $k\ge 3$ is odd and $n=1$,
the frame consisting of $k$ vectors $\{1,1,\dots, 1\}$ in $\mathbb Z_2$ is Parseval, and the Gram matrix $G$ is the $k\times k$ matrix
whose entries are all equal to one. However,  $I-G\equiv I+G$ is the matrix $M$ appearing in Example~\ref{ex:MnotG},
which is not the Gram matrix of a binary Parseval frame. This motivates the search for a condition that
characterizes the existence of complementary Parseval frames. 

\subsection{A simple condition for the existence of complementary Parseval frames}

We observe that if $\mathcal F$ is a Parseval frame with analysis operator $\Theta_\F$ that extends to an orthogonal
matrix, then the column vectors of $\Theta_\F$ are a subset of a set of $n$ orthonormal vectors.
This is true in the binary as well as the real or complex case. Thus, one could try to relate the construction of a complementary Parseval frame to a Gram-Schmidt orthogonalization strategy.
Indeed, this idea allows us to formulate a concrete condition that characterizes when $\mathcal F$ has a complementary Parseval
frame. We prepare this result with a lemma about extending orthonormal sequences.

\begin{lemma}
A binary orthonormal sequence $\mathcal Y = \{v_1, v_2, \dots, v_r\}$ in $\mathbb Z_2^k$ with $r \le k-1$ extends to 
an orthonormal sequence $\{v_1, v_2, \dots, v_k\}$ if and only if $\sum_{i=1}^r v_i \ne \iota_k$,
where $\iota_k$ is the vector in $\Z_2^k$ whose entries are all equal to one.
\end{lemma}
\begin{proof}
If the sequence extends, then $\{v_1, v_2, \dots, v_k\}$ forms a Parseval frame for $\Z_2^k$,
and by the orthonormality, $\sum_{i=1}^k v_i = \sum_{i=1}^k (\iota_k, v_i) v_i =  \iota_k$. On the other hand, the orthonormality
forces the set $\{v_1, v_2, \dots, v_k\}$ to be linearly independent, so $\iota_k$ cannot be
expressed as a linear combination of a proper subset.

To show the converse,
we use an inductive proof. Let $V$ be the analysis operator associated with an orthonormal sequence $
\{v_1, v_2, \dots, v_s\}$, $r \le s \le k-1$ satisfying  $\sum_{i=1}^s v_i \ne \iota_k$.
To extend the sequence by one vector, we need to find $v_{s+1}$ with $(v_{s+1},v_{s+1})=(v_{s+1}, \iota_k) = 1$
and with $(v_{j},v_{s+1}) = 0$ for all $1 \le j \le s$. Using block matrices this is summarized
in the equation
\begin{equation} \label{eq:extension}
\left( \begin{array}{c} V \\ \iota_k^* \end{array} \right) v_{s+1} = \left( \begin{array}{c} 0_s \\ 1 \end{array} \right) \, , 
\end{equation}
where $0_s$ is the zero vector in $\Z_2^s$. 

In order to verify that this equation is consistent, we note that
by the orthonormality of the sequence $\{v_1, v_2, \dots, v_s\}$, the vector $\iota_k$ is  a linear combination 
if and only if $\sum_{i=1}^{s} v_i = \iota_k$. Thus, there exists $v_{s+1}$ which extends the orthonormal sequence.
This is all that is needed if $s=k-1$.



Next, we need to show that if $s\le k-2$, then  a solution $v_{s+1}$ can be chosen so that 
$\sum_{i=1}^{s+1} v_i \ne \iota_k$, so that the iterative extension procedure can be continued.
The solution set of equation (\ref{eq:extension}) forms an affine subspace of $\Z_2^k$  
having dimension $k-(s+1)$, thus contains $2^{k-s-1}$ elements. If $s \le k-2$, then there are
at least two elements in this affine subspace. Consequently, there is one choice of $v_{s+1}$ 
such that   $\sum_{i=1}^{s+1} v_i \ne \iota_k$. 
\end{proof}

We are ready to characterize the complementarity property  for binary Parseval frames.
The condition that determines the existence of a Naimark complement is whether
at least one  frame vector is even, that is, its entries sum to zero.

%
%
\bgb{
\begin{theorem}\label{thm:bpeven}
A binary $(k,n)$-frame $\F$ with $n<k$ has a complementary Parseval frame if and only if 
at least one frame vector is even.
\end{theorem}
\begin{proof}
We first rewrite the condition on the frame vectors in the equivalent form $\Theta_\F \iota_n \ne \iota_k$.
Re-expressed in terms of the column vectors $\{\omega_1, \omega_2, \dots, \omega_n\}$ of $\Theta_\F$, we claim that 
a complementary Parseval frame exists if and only if $\sum_{i=1}^n \omega_n \ne \iota_k$.

On the other hand,
the existence of a complementary Parseval frame is equivalent to
  the sequence of column vectors having an extension to an orthonormal sequence of $k$ elements.
  
  The preceding lemma provides the existence of such an extension if and 
  only if $\sum_{i=1}^n \omega_n \ne \iota_k$, which finishes the proof.
\end{proof}
}

\subsection{A catalog of binary Parseval frames with the complementarity property}\label{sec:catalog}

 
A previous work contained a catalogue of binary Parseval frames for $\mathbb Z_2^n$ when $n$ was small \cite{B09}. 
Here, we wish to compile a list of the binary Parseval frames that have a complementary Parseval frame. 
For notational convenience, we consider $\Theta_\F$ instead of the sequence of frame vectors. 
By Proposition~\ref{prop:compequiv}, every such $\Theta_\F$ is obtained by a selection of columns from a binary orthogonal matrix,
so we could simply list the set of all orthogonal matrices for small $k$. However, such a list quickly becomes extensive
as $k$ increases. To reduce the number of orthogonal matrices, we note that 
although the frame depends on the order in which the columns are selected to form $\Theta_\F$,
the Gram matrix does not. Identifying frames whose Gram matrices coincide
has already been used to avoid repeating information when examining  real or complex frames \cite{B} and 
binary frames~\cite{B09}. We consider an even coarser underlying equivalence relation \cite{GKK01,HP04,BP05} that has also appeared
in the context of binary frames \cite{B09}.

\begin{definition}
Two families $\F=\{f_1, f_2, \dots, f_k\}$ and $\mathcal G = \{g_1, g_2, \dots g_k\}$ in $\mathbb Z_2^n$
are called \emph{switching equivalent} if there is an orthogonal $n\times n$ matrix $U$ and a permutation $\pi$ of the set
$\{1, 2, \dots k\}$ such that
$$
   f_j = U g_{\pi(j)} \mbox{ for all } j \in \{1, 2, \dots \} \, .
$$ 
\end{definition}

Representing the permutation $\pi$ by the associated permutation matrix $P$ with entries $P_{i,j} = \delta_{i,\pi(j)}$
gives that if $\F$ and $\G$ are switching equivalent, then $\Theta_\F = P \Theta_\G U$
with an orthogonal $n \times n$ matrix $U$ and a $k \times k$ permutation matrix $P$.
Alternatively, switching equivalence is stated in the form of an identity for the corresponding Gram matrices.

\begin{theorem}[\cite{B09}]
Two binary $(k,n)$-frames $\F$ and $\G$ are switching equivalent if and only if their
Gram matrices are related by conjugation with a $k\times k$ permutation matrix $P$,
$$ G_\F = P G_\G P^* \, .
$$
\end{theorem}

We deduce a consequence for switching equivalence and Naimark complements, which
is inferred from the role of the Gram matrices in the definition of complementarity.

\begin{corollary}
If $\F$ and $\G$ are switching equivalent binary $(k,n)$-frames, then $\F$ has a Naimark complement if and only if 
$\G$ does.
\end{corollary}
 
 Thus, to provide an exhaustive list, we only need to ensure that
 at least one representative of each switching equivalence class appears as a selection of columns
 in the orthogonal matrices we include.
 To reduce the number of representatives,  we 
identify matrices up to row and column permutations. 

\begin{definition}
Two matrices $A, B \in M_{k,k}(\Z_2)$ are called \emph{permutation equivalent} 
if there are two permutation matrices $P_1, P_2 \in M_{k,k}(\Z_2)$ such that
$A = P_1 B P_2^*$.
\end{definition}

\begin{proposition}
If $U_1$ and $U_2$ are permutation equivalent binary orthogonal matrices,
then each $(k,n)$-frame $\F$ formed by a sequence of $n$ columns of $U_1$ is switching
equivalent to a $(k,n)$-frame $\G$ formed with columns of $U_2$.
\end{proposition}
\begin{proof}
Without loss of generality, we can assume that the analysis matrix $\Theta_\F$ is formed by the
first $n$ columns of $U_1$.  By the equivalence of $U_1$ and $U_2$, $U_1 P_2 = P_1 U_2$
with permutation matrices $P_1$ and $P_2$. The right multiplication of $U_1$ with $P_2$ 
gives a column permutation, which identifies a sequence of columns in $P_1 U_2$ that
is identical to the first $n$ columns of $U_1$. If $\G$ is obtained with the corresponding columns in $U_2$,
then the Gram matrices of $\F$ and $\G$ are related by $G_\F = P_1 G_\G P_1^*$,
which proves the switching equivalence.
\end{proof}

A list of permutation-inequivalent orthogonal $k\times k$ matrices 
allows us to obtain
the Gram matrix of each binary $(k,n)$-frame with a Naimark complement by selecting
an appropriate choice of $n$ columns from an orthogonal $k \times k$ matrix to form $\Theta$ and then by
applying a permutation matrix $P$ to obtain $G_\F = P \Theta \Theta^* P^*$.

Accordingly, each representative of an equivalence class of orthogonal matrices can be chosen
so that the columns
are in lexicographical order. 
\bgb{Table~\ref{tab:orthrep} contains a complete list of representatives of binary orthogonal matrices
for $k \in \{3,4,5,6\}$ from each permutation-equivalence class.
\begin{table}[h]
\centering
\begin{tabular}{|l|llllllllll|}
\hline
\multicolumn{11}{|l|}{Orthogonal $k \times k$ Matrices, $3 \le k \le 6$} \\ \hline
$k$    & \multicolumn{6}{|l}{Integer corresponding to binary column vectors}              &        &    &    &        \\ \hline
3    & 1     & 2     & 4     &        &        &        &        &    &    &        \\ \hline
4    & 1     & 2     & 4     & 8      &        &        &        &    &    &    \\
     & 7     & 11    & 13    & 14     &        &        &        &    &    &        \\ \hline
5    & 1     & 2     & 4     & 8      & 16     &        &        &    &    &        \\
     & 4     & 11    & 19    & 25     & 26     &        &        &    &    &        \\
     & 7     & 8     & 19    & 21     & 22     &        &        &    &    &        \\
     & 7     & 11    & 13    & 14     & 16     &        &        &    &    &        \\ \hline
6    & 1     & 2     & 4     & 8      & 16     & 32     &        &    &    &        \\
     & 4     & 8     & 19    & 35     & 49     & 50     &        &    &    &        \\
     & 4     & 11    & 16    & 35     & 41     & 42     &        &    &    &        \\
     & 4     & 11    & 19    & 25     & 26     & 32     &        &    &    &        \\
     & 7     & 8     & 16    & 35     & 37     & 38     &        &    &    &        \\
     & 7     & 8     & 19    & 21     & 22     & 32     &        &    &    &        \\
     & 7     & 11    & 13    & 14     & 16     & 32     &        &    &    &      \\
     & 13    & 14    & 28    & 44     & 55     & 59     &        &    &    &        \\
     & 21    & 22    & 28    & 47     & 52     & 59     &        &    &    &        \\
     & 25    & 26    & 28    & 47     & 55     & 56     &        &    &    &        \\
     & 31    & 37    & 38    & 44     & 52     & 59     &        &    &    &        \\
     & 31    & 41    & 42    & 44     & 55     & 56     &        &    &    &        \\
     & 31    & 47    & 49    & 50     & 52     & 56     &        &    &    &        \\
     & 31    & 47    & 55    & 59     & 61     & 62     &        &    &    &        \\ \hline
\end{tabular}
\medskip
\caption{Representatives of permutation equivalence classes of orthogonal matrices. 
Up to switching equivalence, the Gram matrix of each binary $(k,n)$-frame with a Naimark complement
is obtained by selecting appropriate columns in one of the listed $k \times k$ orthogonal matrices.}\label{tab:orthrep}
\end{table}
Each column vector in our list is recorded  by the integer obtained from the binary expansion
with the entries of the vector. For example, if a frame vector in $\mathbb Z_2^4$ is $f_1=(1,0,1,1)$,
then it is represented by the integer $2^0+2^2+2^3=13$. Accordingly,
in $\mathbb Z_2^4$, the standard basis is recorded as the sequence of numbers
$1,2,4,8$. }

\section{Gram matrices of binary Parseval frames}\label{sec:binaryGram}

The preceding section on complementarity hinged on the problem that
if $G$ is the Gram matrix of a binary Parseval frame, then $I-G$ may not, although 
it is symmetric and idempotent. \bgb{Again, there is a simple condition that needs to be added;
Gram matrices  of binary Parseval frames are symmetric and idempotent \emph{and} have at least one odd column, that
is, a column whose entries sum to one.
Because of the identity $G^2=G$, having an odd column is equivalent to having a non-zero diagonal entry.
Indeed, it has been shown that for \emph{any} binary symmetric
matrix $G$ without vanishing diagonal, there is a factor $\Theta$ such that $G=\Theta \Theta^*$ and the rank of $\Theta$ 
is equal to that of $G$ \cite{Lem75}. The assumptions needed for our proof are stronger, but our algorithm for producing 
$\Theta$ appears to be more straightforward than the factorization procedure for general symmetric binary matrices. 

}

\bgb{
\begin{theorem}\label{thm:bpg}
A binary symmetric idempotent matrix $M$ is the Gram matrix of a Parseval frame
if and only if it has at least one odd column. 
\end{theorem}
\begin{proof}
First, we re-express the condition on the columns of a symmetric $k \times k$ matrix $M$ in the equivalent
form of the matrix $I_k+M$ having at least one even column or row. This, in turn,
is equivalent to the inequality $(I_k+M)\iota_k \ne \iota_k$.

Next, we recall that both $M$ and $I_k+M$ are assumed to be idempotent.
We observe that any vector $y \in \Z_2^k$ is in the range of an idempotent $P$ if and only if $Py=y$ if and only if
$y$  is in the kernel of $I_k+P$.

Assuming $M$ is the Gram matrix of a Parseval frame, then $M = \Theta \Theta^*$
where $\Theta$ has orthonormal columns and
and $(M +I_k) \Theta = 0$. Combining the two properties gives
$$
   \left( \begin{array}{c} M+I_k \\ \iota_k^* \end{array} \right) \Theta =  \left( \begin{array}{c} 0_{k,n} \\ \iota_n \end{array} \right) \, .
$$
This is inconsistent if and only if $\iota_k$ is in the span of the columns of the idempotent $M+I_k$,
which is equivalent to $(M+I_k)\iota_k = \iota_k$.

Conversely, assuming that $M$ is symmetric and idempotent and that the range of $I_k+M$
does not contain $\iota_k$, we construct a matrix $\Theta$ with orthonormal columns such that $M = \Theta \Theta^*$.
 The assumption on $M$ is equivalent to $M \iota_k \ne 0$,
so at least one row or column of $M$ is odd. Let this column be $\omega_1$, then the fact that $M$ is 
idempotent gives $M\omega_1 = \omega_1$.

Next, we follow an inductive strategy similar to an earlier proof.
What we need is an orthonormal sequence $\{\omega_1, \dots, \omega_{n}\}$
such that $n$ is the rank of $M$ and $M\omega_i =\omega_i$ for all $i \in \{1, 2, \dots, n\}$.
In that case, the range of $M$ is the span of the sequence, and so is the range of $M^*$. 
Thus, the identities $\omega_j^* M \omega_i = \delta_{i,j} = \omega_j^* \Theta \Theta^* \omega_i $ 
for each $i, j \in \{1, 2, \dots, n\}$
imply $M = \Theta \Theta^*$.
 
Proceeding inductively, we need to extend a given orthonormal sequence
$\{\omega_1, \dots, \omega_s\}$ in the kernel of $I_k+M$ by one vector if $s \le n-1$
and if $\iota_k$ is not in the span of the columns of $I_k+M$ combined with the orthonormal sequence.
Let $V$ be a matrix formed by a maximal set of linearly independent rows in $I_k+M$,
then if $M$ has rank $n$, rank nullity gives that $V$ has $k-n$ rows.
Letting $Y$ be the analysis matrix of the orthonormal sequence $\{\omega_1, \dots, \omega_s\}$, then
extending it by one vector requires solving the equation
\begin{equation}\label{eq:VY}
  \left( \begin{array}{c} V\\ Y\\ \iota_k^* \end{array} \right) \omega_{s+1} = \left( \begin{array}{c} 0_{k-n}\\ 0_s\\ 1 \end{array} \right) \, .
\end{equation}

Moreover, in order to guarantee the induction assumption at the next step, we need to show that
if $s \le n-2$, then $\iota_k^*$ is not in the span of the rows of the matrix formed by $V$, $Y$ and $\omega_{s+1}^*$. 
As before, this is obtained by the fact that $V Y^* = 0$, so if $\iota_k = \sum_{i=1}^{s+1} c_i \omega_i + v$
with $v$ being in the span of the columns of $V^*$,
then $Yv=0$ and orthonormality forces $c_i = 1$ for all $i \in \{1, 2, \dots, s+1\}$.
The solutions of the equation (\ref{eq:VY}) form an affine subspace of dimension $k-(k-n)-s-1=n-s-1$, so if $s \le n-2$,
then there are at least two solutions, one of which does not satisfy the sum identity $\iota_k = \sum_{i=1}^{s+1} c_i \omega_i + v$.
\end{proof}
}

\section{Binary cyclic frames and circulant Gram matrices}

Next, we examine a special type of frame whose Gram matrices are circulants.
We recall that a cyclic subspace $V$ of $\mathbb Z_2^k$ has the property that
it is closed under cyclic shifts, that is,
the cyclic shift $S$, which is characterized by $Se_j = e_{j+1 \, (mod)\,  k}$,  
leaves $V$ invariant.

\begin{definition}
A frame $\mathcal F = \{f_1, f_2, \dots, f_k\}$ for $\mathbb Z_2^n$ is called a binary cyclic frame 
if the range of the analysis operator is invariant under the cyclic shift $S$.
If $\mathcal F$ is also Parseval, then we say that is a binary cyclic $(k,n)$-frame.
\end{definition}

Since the range of the Gram matrix $G$ belonging to a Parseval frame is identical to
the set of eigenvectors corresponding to eigenvalue one, we have a simple characterization
of Gram matrices of binary cyclic Parseval frames.

\begin{theorem}
A binary frame $\mathcal F = \{f_1, f_2, \dots, f_k\}$ for $\mathbb Z_2^n$ is a cyclic Parseval frame
if and only if its Gram matrix $G_\F$ is a symmetric idempotent circulant matrix, that is, $G_\F= G_\F^* = G_\F^2$ and $SG_\F S^* = G_\F$,
with only odd column vectors.
\end{theorem}
\begin{proof}
If $G_\F$ is the Gram matrix of a binary cyclic Parseval frame, then from the Parseval property, we know
that $G_\F=G_\F^* = G_\F^2$. Moreover, by the cyclicity of the frame,
the eigenspace
corresponding to eigenvalue one of $G_\F$ is invariant under $S$, and thus if $x=G_\F x$, then $Sx=SG_\F x=G_\F Sx$.
Using this identity repeatedly and writing $y=S^{k-1}x=S^*x$ gives $y=SG_\F S^* y$ for all $y$ in the range of $G_\F$.
By the symmetry of $G_\F$, the range of $G_\F$ is identical to that of $G_\F^*$, so $\langle G_\F x,y\rangle = \langle SG_\F S^* x, y\rangle$
for all $x, y$ in the range of $G_\F$ establishes the circulant property $G_\F=SG_\F S^*$.
If $G_\F$ is a circulant, then each column vector generates all the others 
by applying powers of the cyclic shift to it. Thus, if one column vector is odd, so are all the other column vectors.
Applying Theorem~\ref{thm:bpg} then yields that the Gram matrices of binary cyclic Parseval  frames 
are symmetric idempotent circulant matrices with only odd column vectors.

Conversely, if $G$ is a symmetric idempotent circulant and each column vector is odd, then
Theorem~\ref{thm:bpg} again yields that it is the Gram matrix of a binary Parseval  frame $\F$
with $G = \Theta_\F \Theta_\F^*$.
Moreover, the range of $G$ is invariant under the cyclic shift, because one column vector
generates all the others by applying powers of the cyclic shift to it. 
Since the range of $G$  is identical to that of $\Theta_\F$, $\F$ is a cyclic binary Parseval frame.
\end{proof} 

Since adding the identity matrix changes odd columns of $G$ to even columns, we conclude
that complementary Parseval frames do not exist for binary cyclic Parseval frames. 

\begin{corollary}
If $\mathcal F$ is a binary cyclic Parseval frame, then it has no complementary Parseval frame.  
\end{corollary}

In Table~\ref{tab:circp}, we provide an exhaustive list of the Gram matrices of cyclic binary Parseval frames with $3\le k \le 20$.
Factoring these into the corresponding analysis and synthesis matrices shows that many of these
examples contain repeated frame vectors. 
In an earlier paper, such repeated vectors have been associated with a trivial form of redundancy incorporated in the analysis matrix
$\Theta_\F$.
Table~\ref{tab:circpnontriv} lists the circulant Gram matrices 
of rank $n<k \le 20$
paired with $k \times n$
analysis matrices, for which no repetition of frame vectors occurs.
\begin{table}[bth]
\centering
\begin{adjustbox}{max width=\textwidth}
\begin{tabular}{|l|llllllllllllllllllll|}
\hline
$k$ & \multicolumn{20}{|l|}{First row of circulant $k\times k$ Gram matrix, $3 \le k \le 20$} \\ \hline
3 & 1 & 0 & 0 &  &  &  &  &  &  &  &  &  &  &  &  &  &  &  &  &    \\ \cline{2-21} 
 & 1 & 1 & 1 &  &  &  &  &  &  &  &  &  &  &  &  &  &  &  &  &    \\ \hline
4 & 1 & 0 & 0 & 0 &  &  &  &  &  &  &  &  &  &  &  &  &  &  &  &    \\ \hline
5 & 1 & 0 & 0 & 0 & 0 &  &  &  &  &  &  &  &  &  &  &  &  &  &  &    \\ \cline{2-21} 
 & 1 & 1 & 1 & 1 & 1 &  &  &  &  &  &  &  &  &  &  &  &  &  &  &    \\ \hline
6 & 1 & 0 & 0 & 0 & 0 & 0 &  &  &  &  &  &  &  &  &  &  &  &  &  &    \\ \cline{2-21} 
 & 1 & 0 & 1 & 0 & 1 & 0 &  &  &  &  &  &  &  &  &  &  &  &  &  &    \\ \hline
7 & 1 & 0 & 0 & 0 & 0 & 0 & 0 &  &  &  &  &  &  &  &  &  &  &  &    &  \\ \cline{2-21} 
 & 1 & 1 & 1 & 1 & 1 & 1 & 1 &  &  &  &  &  &  &  &  &  &  &  &  &    \\ \hline
8 & 1 & 0 & 0 & 0 & 0 & 0 & 0 & 0 &  &  &  &  &  &  &  &  &  &  &  &    \\ \hline
9 & 1 & 0 & 0 & 0 & 0 & 0 & 0 & 0 & 0 &  &  &  &  &  &  &  &  &  &  &    \\ \cline{2-21} 
 & 1 & 0 & 0 & 1 & 0 & 0 & 1 & 0 & 0 &  &  &  &  &  &  &  &  &  &  &    \\ \cline{2-21} 
 & 1 & 1 & 1 & 0 & 1 & 1 & 0 & 1 & 1 &  &  &  &  &  &  &  &  &  &  &    \\ \cline{2-21} 
 & 1 & 1 & 1 & 1 & 1 & 1 & 1 & 1 & 1 &  &  &  &  &  &  &  &  &  &  &    \\ \hline
10 & 1 & 0 & 0 & 0 & 0 & 0 & 0 & 0 & 0 & 0 &  &  &  &  &  &  &  &  &  &    \\ \cline{2-21} 
 & 1 & 0 & 1 & 0 & 1 & 0 & 1 & 0 & 1 & 0 &  &  &  &  &  &  &  &  &  &    \\ \hline
11 & 1 & 0 & 0 & 0 & 0 & 0 & 0 & 0 & 0 & 0 & 0 &  &  &  &  &  &  &  &    &  \\ \cline{2-21} 
 & 1 & 1 & 1 & 1 & 1 & 1 & 1 & 1 & 1 & 1 & 1 &  &  &  &  &  &  &  &  &    \\ \hline
12 & 1 & 0 & 0 & 0 & 0 & 0 & 0 & 0 & 0 & 0 & 0 & 0 &  &  &  &  &  &  &  &    \\ \cline{2-21} 
 & 1 & 0 & 0 & 0 & 1 & 0 & 0 & 0 & 1 & 0 & 0 & 0 &  &  &  &  &  &  &  &    \\ \hline
13 & 1 & 0 & 0 & 0 & 0 & 0 & 0 & 0 & 0 & 0 & 0 & 0 & 0 &  &  &  &  &  &  &    \\ \cline{2-21} 
 & 1 & 1 & 1 & 1 & 1 & 1 & 1 & 1 & 1 & 1 & 1 & 1 & 1 &  &  &  &  &  &  &   \\ \hline
14 & 1 & 0 & 0 & 0 & 0 & 0 & 0 & 0 & 0 & 0 & 0 & 0 & 0 & 0 &  &  &  &  &  &    \\ \cline{2-21} 
 & 1 & 0 & 1 & 0 & 1 & 0 & 1 & 0 & 1 & 0 & 1 & 0 & 1 & 0 &  &  &  &  &  &    \\ \hline
15 & 1 & 0 & 0 & 0 & 0 & 0 & 0 & 0 & 0 & 0 & 0 & 0 & 0 & 0 & 0 &  &  &  &  &    \\ \cline{2-21} 
 & 1 & 0 & 0 & 0 & 0 & 1 & 0 & 0 & 0 & 0 & 1 & 0 & 0 & 0 & 0 &  &  &  &  &    \\ \cline{2-21} 
 & 1 & 0 & 0 & 1 & 0 & 0 & 1 & 0 & 0 & 1 & 0 & 0 & 1 & 0 & 0 &  &  &  &  &    \\ \cline{2-21} 
 & 1 & 0 & 0 & 1 & 0 & 1 & 1 & 0 & 0 & 1 & 1 & 0 & 1 & 0 & 0 &  &  &  &  &    \\ \cline{2-21} 
 & 1 & 1 & 1 & 0 & 1 & 0 & 0 & 1 & 1 & 0 & 0 & 1 & 0 & 1 & 1 &  &  &  &  &    \\ \cline{2-21} 
 & 1 & 1 & 1 & 0 & 1 & 1 & 0 & 1 & 1 & 0 & 1 & 1 & 0 & 1 & 1 &  &  &  &  &    \\ \cline{2-21} 
 & 1 & 1 & 1 & 1 & 1 & 0 & 1 & 1 & 1 & 1 & 0 & 1 & 1 & 1 & 1 &  &  &  &  &    \\ \cline{2-21} 
 & 1 & 1 & 1 & 1 & 1 & 1 & 1 & 1 & 1 & 1 & 1 & 1 & 1 & 1 & 1 &  &  &  &  &    \\ \hline
16 & 1 & 0 & 0 & 0 & 0 & 0 & 0 & 0 & 0 & 0 & 0 & 0 & 0 & 0 & 0 & 0 &  &  &  &    \\ \hline
17 & 1 & 0 & 0 & 0 & 0 & 0 & 0 & 0 & 0 & 0 & 0 & 0 & 0 & 0 & 0 & 0 & 0 &  &  &    \\ \cline{2-21} 
 & 1 & 0 & 0 & 1 & 0 & 1 & 1 & 1 & 0 & 0 & 1 & 1 & 1 & 0 & 1 & 0 & 0 &  &  &    \\ \cline{2-21} 
 & 1 & 1 & 1 & 0 & 1 & 0 & 0 & 0 & 1 & 1 & 0 & 0 & 0 & 1 & 0 & 1 & 1 &  &  &    \\ \cline{2-21} 
 & 1 & 1 & 1 & 1 & 1 & 1 & 1 & 1 & 1 & 1 & 1 & 1 & 1 & 1 & 1 & 1 & 1 &  &  &    \\ \hline
18 & 1 & 0 & 0 & 0 & 0 & 0 & 0 & 0 & 0 & 0 & 0 & 0 & 0 & 0 & 0 & 0 & 0 & 0 &  &    \\ \cline{2-21} 
 & 1 & 0 & 0 & 0 & 0 & 0 & 1 & 0 & 0 & 0 & 0 & 0 & 1 & 0 & 0 & 0 & 0 & 0 &  &    \\ \cline{2-21} 
 & 1 & 0 & 1 & 0 & 1 & 0 & 0 & 0 & 1 & 0 & 1 & 0 & 0 & 0 & 1 & 0 & 1 & 0 &  &    \\ \cline{2-21} 
 & 1 & 0 & 1 & 0 & 1 & 0 & 1 & 0 & 1 & 0 & 1 & 0 & 1 & 0 & 1 & 0 & 1 & 0 &  &    \\ \hline
19 & 1 & 0 & 0 & 0 & 0 & 0 & 0 & 0 & 0 & 0 & 0 & 0 & 0 & 0 & 0 & 0 & 0 & 0 & 0 &    \\ \cline{2-21} 
 & 1 & 1 & 1 & 1 & 1 & 1 & 1 & 1 & 1 & 1 & 1 & 1 & 1 & 1 & 1 & 1 & 1 & 1 & 1 &    \\ \hline
20 & 1 & 0 & 0 & 0 & 0 & 0 & 0 & 0 & 0 & 0 & 0 & 0 & 0 & 0 & 0 & 0 & 0 & 0 & 0 & 0   \\ \cline{2-21} 
 & 1 & 0 & 0 & 0 & 1 & 0 & 0 & 0 & 1 & 0 & 0 & 0 & 1 & 0 & 0 & 0 & 1 & 0 & 0 & 0   \\ \hline
\end{tabular}
\end{adjustbox}
\medskip
\caption{Circulant Gram matrices of binary cyclic $(k,n)$-frames, first row shown.}\label{tab:circp}
\end{table}

\begin{table}[h]
\centering
\fbox{
\begin{adjustbox}{max width=.94\textwidth}
\begin{tabular}{llllllllllllllllllll|llllllllllllllllll}
& & &  & & & & & \multicolumn{11}{l}{Circulant Gram matrix} &  &  & \multicolumn{17}{l}{Corresponding} \\
 &  &  &  &  &  &  &  &  &  &  &  &  &  &  &  &  &  &  &  &  & \multicolumn{10}{l}{analysis matrix} &  &  &  &  &  &  &    \\
 &  &  &  &  &  &  &  &  &  &  &  &  &  &  &  &  &  &  &  &  &  &  &  &  &  &  &  &  &  &  &  &  &  &  &  &  &  \\
 &  &  &  &  &  &  &  &  &  & 1 & 1 & 1 & 0 & 1 & 1 & 0 & 1 & 1 &  &  & 1 & 0 & 0 & 1 & 1 & 1 & 1 &  &  &  &  &  &  &  &  &  &  \\
 &  &  &  &  &  &  &  &  &  & 1 & 1 & 1 & 1 & 0 & 1 & 1 & 0 & 1 &  &  & 1 & 1 & 1 & 0 & 0 & 1 & 1 &  &  &  &  &  &  &  &  &  &  \\
 &  &  &  &  &  &  &  &  &  & 1 & 1 & 1 & 1 & 1 & 0 & 1 & 1 & 0 &  &  & 1 & 1 & 1 & 1 & 1 & 0 & 0 &  &  &  &  &  &  &  &  &  &  \\
 &  &  &  & \multicolumn{4}{l}{$k=9$} &  &  & 0 & 1 & 1 & 1 & 1 & 1 & 0 & 1 & 1 &  &  & 0 & 1 & 0 & 1 & 1 & 1 & 1 &  &  &  &  &  &  &  &  &  &  \\
 &  &  &  &  &  &  &  &  &  & 1 & 0 & 1 & 1 & 1 & 1 & 1 & 0 & 1 &  &  & 0 & 0 & 0 & 1 & 0 & 1 & 1 &  &  &  &  &  &  &  &  &  &  \\
 &  &  &  &  &  &  &  &  &  & 1 & 1 & 0 & 1 & 1 & 1 & 1 & 1 & 0 &  &  & 0 & 0 & 0 & 0 & 0 & 1 & 0 &  &  &  &  &  &  &  &  &  &  \\
 &  &  &  &  &  &  &  &  &  & 0 & 1 & 1 & 0 & 1 & 1 & 1 & 1 & 1 &  &  & 0 & 0 & 1 & 1 & 1 & 1 & 1 &  &  &  &  &  &  &  &  &  &  \\
 &  &  &  &  &  &  &  &  &  & 1 & 0 & 1 & 1 & 0 & 1 & 1 & 1 & 1 &  &  & 0 & 0 & 0 & 0 & 1 & 1 & 1 &  &  &  &  &  &  &  &  &  &  \\
 &  &  &  &  &  &  &  &  &  & 1 & 1 & 0 & 1 & 1 & 0 & 1 & 1 & 1 &  &  & 0 & 0 & 0 & 0 & 0 & 0 & 1 &  &  &  &  &  &  &  &  &  &  \\
 &  &  &  &  &  &  &  &  &  &  &  &  &  &  &  &  &  &  &  &  &  &  &  &  &  &  &  &  &  &  &  &  &  &  &  &  &  \\
 &  &  &  & 1 & 1 & 1 & 0 & 1 & 0 & 0 & 1 & 1 & 0 & 0 & 1 & 0 & 1 & 1 &  &  & 1 & 0 & 1 & 0 & 1 & 1 & 1 &  &  &  &  &  &  &  &  &  &  \\
 &  &  &  & 1 & 1 & 1 & 1 & 0 & 1 & 0 & 0 & 1 & 1 & 0 & 0 & 1 & 0 & 1 &  &  & 1 & 1 & 0 & 1 & 1 & 0 & 1 &  &  & \multicolumn{5}{l}{$n=7$} &  &  &  \\
 &  &  &  & 1 & 1 & 1 & 1 & 1 & 0 & 1 & 0 & 0 & 1 & 1 & 0 & 0 & 1 & 0 &  &  & 0 & 1 & 0 & 1 & 1 & 0 & 0 &  &  &  &  &  &  &  &  &  &  \\
 &  &  &  & 0 & 1 & 1 & 1 & 1 & 1 & 0 & 1 & 0 & 0 & 1 & 1 & 0 & 0 & 1 &  &  & 1 & 0 & 1 & 1 & 0 & 1 & 1 &  &  &  &  &  &  &  &  &  &  \\
 &  &  &  & 1 & 0 & 1 & 1 & 1 & 1 & 1 & 0 & 1 & 0 & 0 & 1 & 1 & 0 & 0 &  &  & 1 & 1 & 0 & 1 & 1 & 1 & 0 &  &  &  &  &  &  &  &  &  &  \\
 &  &  &  & 0 & 1 & 0 & 1 & 1 & 1 & 1 & 1 & 0 & 1 & 0 & 0 & 1 & 1 & 0 &  &  & 1 & 1 & 0 & 0 & 1 & 0 & 0 &  &  &  &  &  &  &  &  &  &  \\
 &  &  &  & 0 & 0 & 1 & 0 & 1 & 1 & 1 & 1 & 1 & 0 & 1 & 0 & 0 & 1 & 1 &  &  & 0 & 1 & 1 & 0 & 0 & 0 & 1 &  &  &  &  &  &  &  &  &  &  \\
 &  &  &  & 1 & 0 & 0 & 1 & 0 & 1 & 1 & 1 & 1 & 1 & 0 & 1 & 0 & 0 & 1 &  &  & 1 & 0 & 0 & 0 & 0 & 1 & 1 &  &  &  &  &  &  &  &  &  &  \\
 &  &  &  & 1 & 1 & 0 & 0 & 1 & 0 & 1 & 1 & 1 & 1 & 1 & 0 & 1 & 0 & 0 &  &  & 1 & 1 & 0 & 1 & 0 & 0 & 0 &  &  &  &  &  &  &  &  &  &  \\
 &  &  &  & 0 & 1 & 1 & 0 & 0 & 1 & 0 & 1 & 1 & 1 & 1 & 1 & 0 & 1 & 0 &  &  & 0 & 1 & 1 & 0 & 0 & 1 & 0 &  &  &  &  &  &  &  &  &  &  \\
 &  &  &  & 0 & 0 & 1 & 1 & 0 & 0 & 1 & 0 & 1 & 1 & 1 & 1 & 1 & 0 & 1 &  &  & 0 & 0 & 0 & 1 & 0 & 1 & 1 &  &  &  &  &  &  &  &  &  &  \\
 &  &  &  & 1 & 0 & 0 & 1 & 1 & 0 & 0 & 1 & 0 & 1 & 1 & 1 & 1 & 1 & 0 &  &  & 0 & 0 & 0 & 0 & 0 & 1 & 0 &  &  &  &  &  &  &  &  &  &  \\
 &  &  &  & 0 & 1 & 0 & 0 & 1 & 1 & 0 & 0 & 1 & 0 & 1 & 1 & 1 & 1 & 1 &  &  & 0 & 0 & 1 & 1 & 1 & 1 & 1 &  &  &  &  &  &  &  &  &  &  \\
 &  &  &  & 1 & 0 & 1 & 0 & 0 & 1 & 1 & 0 & 0 & 1 & 0 & 1 & 1 & 1 & 1 &  &  & 0 & 0 & 0 & 0 & 1 & 1 & 1 &  &  &  &  &  &  &  &  &  &  \\
 &  &  &  & 1 & 1 & 0 & 1 & 0 & 0 & 1 & 1 & 0 & 0 & 1 & 0 & 1 & 1 & 1 &  &  & 0 & 0 & 0 & 0 & 0 & 0 & 1 &  &  &  &  &  &  &  &  &  &  \\
 &  &  &  &  &  &  &  &  &  &  &  &  &  &  &  &  &  &  &  &  &  &  &  &  &  &  &  &  &  &  &  &  &  &  &  &  &  \\
 &  &  &  & 1 & 0 & 0 & 1 & 0 & 1 & 1 & 0 & 0 & 1 & 1 & 0 & 1 & 0 & 0 &  &  & 1 & 0 & 1 & 0 & 0 & 0 & 1 & 0 & 0 &  &  &  &  &  &  &  &  \\
 &  &  &  & 0 & 1 & 0 & 0 & 1 & 0 & 1 & 1 & 0 & 0 & 1 & 1 & 0 & 1 & 0 &  &  & 0 & 1 & 0 & 0 & 1 & 0 & 0 & 1 & 0 &  &  &  &  &  &  &  &  \\
 &  &  &  & 0 & 0 & 1 & 0 & 0 & 1 & 0 & 1 & 1 & 0 & 0 & 1 & 1 & 0 & 1 &  &  & 1 & 0 & 0 & 0 & 0 & 0 & 1 & 0 & 1 &  &  &  &  &  &  &  &  \\
 &  &  &  & 1 & 0 & 0 & 1 & 0 & 0 & 1 & 0 & 1 & 1 & 0 & 0 & 1 & 1 & 0 &  &  & 1 & 1 & 1 & 0 & 0 & 0 & 1 & 1 & 0 &  &  &  &  &  &  &  &  \\
 &  &  &  & 0 & 1 & 0 & 0 & 1 & 0 & 0 & 1 & 0 & 1 & 1 & 0 & 0 & 1 & 1 &  &  & 1 & 1 & 1 & 1 & 1 & 0 & 0 & 1 & 1 &  &  &  &  &  &  &  &  \\
 &  &  &  & 1 & 0 & 1 & 0 & 0 & 1 & 0 & 0 & 1 & 0 & 1 & 1 & 0 & 0 & 1 &  &  & 0 & 1 & 1 & 0 & 1 & 1 & 0 & 0 & 1 &  &  & \multicolumn{5}{l}{$n=9$} &  \\
 &  &  &  & 1 & 1 & 0 & 1 & 0 & 0 & 1 & 0 & 0 & 1 & 0 & 1 & 1 & 0 & 0 &  &  & 1 & 0 & 1 & 1 & 1 & 0 & 1 & 0 & 0 &  &  &  &  &  &  &  &  \\
 &  &  &  & 0 & 1 & 1 & 0 & 1 & 0 & 0 & 1 & 0 & 0 & 1 & 0 & 1 & 1 & 0 &  &  & 0 & 1 & 1 & 1 & 1 & 1 & 1 & 1 & 0 &  &  &  &  &  &  &  &  \\
 &  &  &  & 0 & 0 & 1 & 1 & 0 & 1 & 0 & 0 & 1 & 0 & 0 & 1 & 0 & 1 & 1 &  &  & 0 & 0 & 0 & 1 & 1 & 1 & 0 & 1 & 1 &  &  &  &  &  &  &  &  \\
 &  &  &  & 1 & 0 & 0 & 1 & 1 & 0 & 1 & 0 & 0 & 1 & 0 & 0 & 1 & 0 & 1 &  &  & 0 & 0 & 0 & 0 & 0 & 1 & 1 & 0 & 1 &  &  &  &  &  &  &  &  \\
 &  &  &  & 1 & 1 & 0 & 0 & 1 & 1 & 0 & 1 & 0 & 0 & 1 & 0 & 0 & 1 & 0 &  &  & 0 & 0 & 1 & 1 & 0 & 0 & 0 & 1 & 0 &  &  &  &  &  &  &  &  \\
 &  &  &  & 0 & 1 & 1 & 0 & 0 & 1 & 1 & 0 & 1 & 0 & 0 & 1 & 0 & 0 & 1 &  &  & 0 & 0 & 0 & 0 & 1 & 1 & 0 & 0 & 1 &  &  &  &  &  &  &  &  \\
 &  &  &  & 1 & 0 & 1 & 1 & 0 & 0 & 1 & 1 & 0 & 1 & 0 & 0 & 1 & 0 & 0 &  &  & 0 & 0 & 0 & 0 & 0 & 0 & 1 & 0 & 0 &  &  &  &  &  &  &  &  \\
 &  &  &  & 0 & 1 & 0 & 1 & 1 & 0 & 0 & 1 & 1 & 0 & 1 & 0 & 0 & 1 & 0 &  &  & 0 & 0 & 0 & 0 & 0 & 0 & 0 & 1 & 0 &  &  &  &  &  &  &  &  \\
 &  &  &  & 0 & 0 & 1 & 0 & 1 & 1 & 0 & 0 & 1 & 1 & 0 & 1 & 0 & 0 & 1 &  &  & 0 & 0 & 0 & 0 & 0 & 0 & 0 & 0 & 1 &  &  &  &  &  &  &  &  \\
\multicolumn{4}{l}{$k=15$} &  &  &  &  &  &  &  &  &  &  &  &  &  &  &  &  &  &  &  &  &  &  &  &  &  &  &  &  &  &  &  &  &  &  \\
 &  &  &  & 1 & 1 & 1 & 1 & 1 & 0 & 1 & 1 & 1 & 1 & 0 & 1 & 1 & 1 & 1 &  &  & 1 & 0 & 0 & 1 & 1 & 1 & 1 & 1 & 1 & 1 & 1 &  &  &  &  &  &  \\
 &  &  &  & 1 & 1 & 1 & 1 & 1 & 1 & 0 & 1 & 1 & 1 & 1 & 0 & 1 & 1 & 1 &  &  & 1 & 1 & 1 & 0 & 0 & 1 & 1 & 1 & 1 & 1 & 1 &  &  &  &  &  &  \\
 &  &  &  & 1 & 1 & 1 & 1 & 1 & 1 & 1 & 0 & 1 & 1 & 1 & 1 & 0 & 1 & 1 &  &  & 1 & 1 & 1 & 1 & 1 & 0 & 0 & 1 & 1 & 1 & 1 &  &  &  &  &  &  \\
 &  &  &  & 1 & 1 & 1 & 1 & 1 & 1 & 1 & 1 & 0 & 1 & 1 & 1 & 1 & 0 & 1 &  &  & 1 & 1 & 1 & 1 & 1 & 1 & 1 & 0 & 0 & 1 & 1 &  &  &  &  &  &  \\
 &  &  &  & 1 & 1 & 1 & 1 & 1 & 1 & 1 & 1 & 1 & 0 & 1 & 1 & 1 & 1 & 0 &  &  & 1 & 1 & 1 & 1 & 1 & 1 & 1 & 1 & 1 & 0 & 0 &  &  &  &  &  &  \\
 &  &  &  & 0 & 1 & 1 & 1 & 1 & 1 & 1 & 1 & 1 & 1 & 0 & 1 & 1 & 1 & 1 &  &  & 0 & 1 & 0 & 1 & 1 & 1 & 1 & 1 & 1 & 1 & 1 &  &  &  &  &  &  \\
 &  &  &  & 1 & 0 & 1 & 1 & 1 & 1 & 1 & 1 & 1 & 1 & 1 & 0 & 1 & 1 & 1 &  &  & 0 & 0 & 0 & 1 & 0 & 1 & 1 & 1 & 1 & 1 & 1 &  &  &  &  &  &  \\
 &  &  &  & 1 & 1 & 0 & 1 & 1 & 1 & 1 & 1 & 1 & 1 & 1 & 1 & 0 & 1 & 1 &  &  & 0 & 0 & 0 & 0 & 0 & 1 & 0 & 1 & 1 & 1 & 1 &  &  & \multicolumn{3}{l}{$n=11$} &  \\
 &  &  &  & 1 & 1 & 1 & 0 & 1 & 1 & 1 & 1 & 1 & 1 & 1 & 1 & 1 & 0 & 1 &  &  & 0 & 0 & 0 & 0 & 0 & 0 & 0 & 1 & 0 & 1 & 1 &  &  &  &  &  &  \\
 &  &  &  & 1 & 1 & 1 & 1 & 0 & 1 & 1 & 1 & 1 & 1 & 1 & 1 & 1 & 1 & 0 &  &  & 0 & 0 & 0 & 0 & 0 & 0 & 0 & 0 & 0 & 1 & 0 &  &  &  &  &  &  \\
 &  &  &  & 0 & 1 & 1 & 1 & 1 & 0 & 1 & 1 & 1 & 1 & 1 & 1 & 1 & 1 & 1 &  &  & 0 & 0 & 1 & 1 & 1 & 1 & 1 & 1 & 1 & 1 & 1 &  &  &  &  &  &  \\
 &  &  &  & 1 & 0 & 1 & 1 & 1 & 1 & 0 & 1 & 1 & 1 & 1 & 1 & 1 & 1 & 1 &  &  & 0 & 0 & 0 & 0 & 1 & 1 & 1 & 1 & 1 & 1 & 1 &  &  &  &  &  &  \\
 &  &  &  & 1 & 1 & 0 & 1 & 1 & 1 & 1 & 0 & 1 & 1 & 1 & 1 & 1 & 1 & 1 &  &  & 0 & 0 & 0 & 0 & 0 & 0 & 1 & 1 & 1 & 1 & 1 &  &  &  &  &  &  \\
 &  &  &  & 1 & 1 & 1 & 0 & 1 & 1 & 1 & 1 & 0 & 1 & 1 & 1 & 1 & 1 & 1 &  &  & 0 & 0 & 0 & 0 & 0 & 0 & 0 & 0 & 1 & 1 & 1 &  &  &  &  &  &  \\
 &  &  &  & 1 & 1 & 1 & 1 & 0 & 1 & 1 & 1 & 1 & 0 & 1 & 1 & 1 & 1 & 1 &  &  & 0 & 0 & 0 & 0 & 0 & 0 & 0 & 0 & 0 & 0 & 1 &  &  &  &  &  &  \\
 &  &  &  &  &  &  &  &  &  &  &  &  &  &  &  &  &  &  &  &  &  &  &  &  &  &  &  &  &  &  &  &  &  &  &  &  &  \\
 &  &  &  & 1 & 1 & 1 & 0 & 1 & 1 & 0 & 1 & 1 & 0 & 1 & 1 & 0 & 1 & 1 &  &  & 1 & 0 & 0 & 1 & 1 & 1 & 1 & 0 & 0 & 1 & 1 & 1 & 1 &  &  &  &  \\
 &  &  &  & 1 & 1 & 1 & 1 & 0 & 1 & 1 & 0 & 1 & 1 & 0 & 1 & 1 & 0 & 1 &  &  & 1 & 1 & 1 & 0 & 0 & 1 & 1 & 1 & 1 & 0 & 0 & 1 & 1 &  &  &  &  \\
 &  &  &  & 1 & 1 & 1 & 1 & 1 & 0 & 1 & 1 & 0 & 1 & 1 & 0 & 1 & 1 & 0 &  &  & 1 & 1 & 1 & 1 & 1 & 0 & 0 & 1 & 1 & 1 & 1 & 0 & 0 &  &  &  &  \\
 &  &  &  & 0 & 1 & 1 & 1 & 1 & 1 & 0 & 1 & 1 & 0 & 1 & 1 & 0 & 1 & 1 &  &  & 0 & 1 & 0 & 1 & 1 & 1 & 1 & 0 & 0 & 1 & 1 & 1 & 1 &  &  &  &  \\
 &  &  &  & 1 & 0 & 1 & 1 & 1 & 1 & 1 & 0 & 1 & 1 & 0 & 1 & 1 & 0 & 1 &  &  & 0 & 0 & 0 & 1 & 0 & 1 & 1 & 1 & 1 & 0 & 0 & 1 & 1 &  &  &  &  \\
 &  &  &  & 1 & 1 & 0 & 1 & 1 & 1 & 1 & 1 & 0 & 1 & 1 & 0 & 1 & 1 & 0 &  &  & 0 & 0 & 0 & 0 & 0 & 1 & 0 & 1 & 1 & 1 & 1 & 0 & 0 &  &  &  &  \\
 &  &  &  & 0 & 1 & 1 & 0 & 1 & 1 & 1 & 1 & 1 & 0 & 1 & 1 & 0 & 1 & 1 &  &  & 0 & 0 & 1 & 1 & 1 & 1 & 1 & 0 & 0 & 1 & 1 & 1 & 1 &  &  &  &  \\
 &  &  &  & 1 & 0 & 1 & 1 & 0 & 1 & 1 & 1 & 1 & 1 & 0 & 1 & 1 & 0 & 1 &  &  & 0 & 0 & 0 & 0 & 1 & 1 & 1 & 1 & 1 & 0 & 0 & 1 & 1 &  & \multicolumn{3}{l}{$n=13$} \\
 &  &  &  & 1 & 1 & 0 & 1 & 1 & 0 & 1 & 1 & 1 & 1 & 1 & 0 & 1 & 1 & 0 &  &  & 0 & 0 & 0 & 0 & 0 & 0 & 1 & 1 & 1 & 1 & 1 & 0 & 0 &  &  &  &  \\
 &  &  &  & 0 & 1 & 1 & 0 & 1 & 1 & 0 & 1 & 1 & 1 & 1 & 1 & 0 & 1 & 1 &  &  & 0 & 0 & 0 & 0 & 0 & 0 & 0 & 1 & 0 & 1 & 1 & 1 & 1 &  &  &  &  \\
 &  &  &  & 1 & 0 & 1 & 1 & 0 & 1 & 1 & 0 & 1 & 1 & 1 & 1 & 1 & 0 & 1 &  &  & 0 & 0 & 0 & 0 & 0 & 0 & 0 & 0 & 0 & 1 & 0 & 1 & 1 &  &  &  &  \\
 &  &  &  & 1 & 1 & 0 & 1 & 1 & 0 & 1 & 1 & 0 & 1 & 1 & 1 & 1 & 1 & 0 &  &  & 0 & 0 & 0 & 0 & 0 & 0 & 0 & 0 & 0 & 0 & 0 & 1 & 0 &  &  &  &  \\
 &  &  &  & 0 & 1 & 1 & 0 & 1 & 1 & 0 & 1 & 1 & 0 & 1 & 1 & 1 & 1 & 1 &  &  & 0 & 0 & 0 & 0 & 0 & 0 & 0 & 0 & 1 & 1 & 1 & 1 & 1 &  &  &  &  \\
 &  &  &  & 1 & 0 & 1 & 1 & 0 & 1 & 1 & 0 & 1 & 1 & 0 & 1 & 1 & 1 & 1 &  &  & 0 & 0 & 0 & 0 & 0 & 0 & 0 & 0 & 0 & 0 & 1 & 1 & 1 &  &  &  &  \\
 &  &  &  & 1 & 1 & 0 & 1 & 1 & 0 & 1 & 1 & 0 & 1 & 1 & 0 & 1 & 1 & 1 &  &  & 0 & 0 & 0 & 0 & 0 & 0 & 0 & 0 & 0 & 0 & 0 & 0 & 1 &  &  &  & 
\end{tabular}
\end{adjustbox}
}
\medskip
\caption{Gram and analysis matrices of binary cyclic $(k,n)$-frames, $k>n$, whose vectors do not repeat.}\label{tab:circpnontriv}
\end{table}


\end{document}